\documentclass[10pt]{article}
\usepackage[square,numbers]{natbib}
\usepackage{amsthm}
\usepackage{authblk}
\usepackage{graphicx}
\usepackage{amssymb}
\usepackage{amsthm}
\usepackage{amsmath}
\usepackage{mathtools}
\usepackage{comment}

\usepackage[T1]{fontenc}
\usepackage[utf8]{inputenc}
\usepackage{tabularx,ragged2e,booktabs,caption}

\newtheorem{theorem}{Theorem}[section]
\newtheorem{corollary}{Corollary}[theorem]
\newtheorem{lemma}[theorem]{Lemma}

\def\cL{{\mathcal L}}
\def\cM{{\mathcal M}}

\def\cR{{\mathcal R}}

\def\cU{{\mathcal U}}

\begin{document}

\title{ Partial Gaussian sums and the P\'{o}lya--Vinogradov inequality for primitive characters}
\author{Matteo Bordignon}
\affil[]{School of Science, University of New South Wales Canberra }
\affil{m.bordignon@student.unsw.edu.au}
\date{\vspace{-5ex}}
\maketitle

\begin{abstract}
In this paper we obtain a new fully explicit constant for the P\'olya-Vinogradov inequality for primitive characters. Given a primitive character $\chi$ modulo $q$, we prove the following upper bound
\begin{align*}
\left| \sum_{1 \le n\le N} \chi(n) \right|\le c \sqrt{q} \log q,
\end{align*}
where $c=3/(4\pi^2)+o_q(1)$ for even characters and $c=3/(8\pi)+o_q(1)$ for odd characters, with explicit  $o_q(1)$ terms. This improves a result of Frolenkov and Soundararajan for large $q$. We proceed, following Hildebrand, obtaining the explicit version of a result by Montgomery--Vaughan on partial Gaussian sums and an explicit Burgess-like result on convoluted Dirichlet characters.
\end{abstract}
\section{Introduction}
It is of high interest studying the upper bound of the following quantity
\begin{equation*}
S(N, \chi):=\left|\sum_{n=1}^N \chi(n) \right|,
\end{equation*}
with $N\in \mathbb{N}$ and $\chi$ a non-principal Dirichlet character modulo $q$. The famous P\'{o}lya--Vinogradov inequality tells us that 
\begin{equation*}
S(N, \chi)\ll \sqrt{q}\log q,
\end{equation*}
and aside for the implied constant, this is the best known result. 
The focus is now on the implied constant, with a distinction between asymptotically explicit and completely explicit results. The best asymptotic constant can be found in the papers by Hildebrand \cite{Hild1} and Granville and Soundararajan \cite{G-S}. The explicit results have generally worst leading terms, the exception is for primitive characters of square-free moduli for which the author and Kerr \cite{B-K} proved a result that is comparable with the asymptotic one. There have been many completely explicit results, we will be focussing on primitive characters, as these results can be easily  extended to all non-principal characters. All the late results have the following shape
\begin{equation}
\label{eq:oldPV}
|S(N,\chi)|\le \begin{cases}\frac{1}{\pi^2}\sqrt{q}\log q+\delta_1 \sqrt{q}\log \log q+\delta_2\sqrt{q}~~\text{for}~~\chi(-1)=1, \\\frac{1}{2\pi}\sqrt{q}\log q+\delta_3 \sqrt{q}\log \log q+\delta_4\sqrt{q}~~\text{for}~~\chi(-1)=-1, \end{cases}
\end{equation}
with the second constants improving as follows:
\begin{itemize}
\item $\delta_1=\frac{2}{\pi^2}, \delta_2=\frac{3}{4}, \delta_3= \frac{1}{\pi}~ \text{and}~ \delta_4=1$ by Pomerance \cite{Pomerance},
\item Frolenkov \cite{Frolenkov} proves that for certain values of $\delta_2$ and $\delta_4$ it is possible to take $\delta_1=\delta_3= 0$,
\item Frolenkov and Soundararajan \cite{F-S} further improve the result showing that it is possible to take $\delta_2=\frac{1}{2}$, for $q\ge 1200$ and $\delta_4=1$, for $q\ge 40$.
\end{itemize}
The improvements above are on the constants of the remainder terms. Our aim is to improve on the leading constant using Hildebrand's approach \cite{Hildebrand}, that relies on two results: an upper bound on partial Gaussian sums due to Montgomery and Vaughan and the version of the Burgess bound for all non-principal characters from \cite{Burgess3}.\\
We start proving the following explicit version of Corollary 1 \cite{M-V} by Montgomery and Vaughan. Let $B \ge 1$ be a constant, and $F$ be the class of all multiplicative functions $f$ such that
\begin{equation}
\label{2}
| f(n)|\le B.
\end{equation}
With $f \in F$, $\alpha$ real and $e(\alpha)=\exp(2\pi i \alpha)$ write
\begin{equation*}
S(\alpha)=\sum^{N}_{n=1} f(n)e(n\alpha).
\end{equation*}
\begin{corollary}
\label{c1}
Suppose that $|\alpha - a/q| \le q^{-2}$, $(a,q)=1$, $E\ge 4$ with $k\ge 2$ and $\frac{e^3}{E} \le R\le q \le N/R$. Then
\begin{equation}
\label{SA}
S(\alpha) \le c_1(B,E,R) \frac{N}{\log N}+ c_2(B,E,R)  \frac{N\log^{\frac{3}{2}} ER}{\sqrt{R}},
\end{equation}
with the functions $c_1$ and $c_2$ defined in Theorem \ref{P-V}.
\end{corollary}	
Note that condition \eqref{2} simplifies computations compared to $\sum_{n=1}^N| f(n)|\le B^2 N$ found in \cite{M-V}.\\
Proving an explicit version of the Burgess's bound in \cite{Burgess3} is difficult, but the following result, that is an explicit Burgess-like result on convoluted Dirichlet characters, is enough for our purposes.
\begin{theorem}
\label{TB}
Let $q$ and $k$ be integers and $h$ and $m$ real positive numbers. Assume that $q> (hk)^4$. Let $\chi$ be a primitive character mod $q$  and $\psi$ be any character mod $k$. For any integers $M$ and $N<q$ we have
\begin{align*}
\left|\sum_{M<n\le M+N}\psi(n)\chi(n)\right|\le m k d(q)^{3/2}N^{1/2}q^{3/16}(\log q \log \log q)^{\frac{1}{2}},
\end{align*}
with $d$ the divisor counting function and $q\ge q_0$, $h$ and $m$ in the following table.
\end{theorem}
\begin{center}
\captionof{table}{Small  and large $q$} \label{tab:m} 
\begin{tabular}{| c | c | c | c | c | }
  \hline
  $m$ & $\log_{10} q_0$ & $h$  \\[0.5ex] 
  \hline  \hline
   2.29 & 5& 10 \\
  \hline
   1.55 & 10 & 20 \\
  \hline
  1.29 & 15 & 30  \\
  \hline
  1.15& 20 & 40  \\
  \hline
\end{tabular}
\quad 
\begin{tabular}{| c | c | c | c | c | }
  \hline
  $m$ & $\log_{10} q_0$  & $h$  \\[0.5ex] 
  \hline  \hline
   0.73 &100 & 40 \\
  \hline
  0.61 & 200 & 40 \\
  \hline
  0.55 &300  &  200 \\
  \hline
  0.53& 400 & 300  \\
  \hline
\end{tabular}
\end{center}
If we restrict to $q$ prime, we should be able to improve the above result, but as we are mainly interested in a result for any $q$ we will not further exploit this possibility.
Related explicit results can be found in \cite{B-K}, \cite{Forrest} and \cite{Trev}.
Using the above result we are able to relax the conditions on $\alpha$ that appear in Corollary \ref{c1}, thus obtaining the following fundamental result.
\begin{lemma}
\label{LI1}
Take any $h$,~$m$ and lower bound for $q$ in Table \ref{tab:m}. Take $x$ such that $q^{\frac{3}{8}+\epsilon }\le x \le q$ and, fixed a real $\gamma\ge 2$, for any $q$ such that $ h(\log q)^{\gamma}\le  q^{\frac{1}{4}}$ and $E\ge 4$. We have, uniformly for all primitive characters $\chi$  modulo $q$ as above,
\begin{equation*}
\left| \sum_{n \le x} \chi(n) e(\alpha n) \right| \le c(1, q, \gamma, \epsilon, m) \frac{x}{\log q},
\end{equation*}
with the function $c$ defined in Theorem \ref{P-V}.
\end{lemma}
This will give us the desired Theorem \ref{P-V}.
The problem is now reduced to a computational one, we need to minimize $\epsilon$, $q$ and $c(E, q, \gamma, \epsilon)$. We will thus obtain the following result, see Section \ref{OP} for more details.
\begin{theorem}
\label{EP-V}
With $\chi$ a primitive Dirichlet character modulo $q$ we have 
\begin{align*}
  \left|S(N, \chi)
   \right| \le 
  \begin{dcases}
     \frac{2}{\pi^2}(\frac{3}{8}+\epsilon)\sqrt{q}\log q +h_1(E, q, \gamma, \epsilon,m)\sqrt{q}~~\text{if}~~\chi(-1)=1,\\
    ~\\
   \frac{1}{\pi} (\frac{3}{8}+\epsilon)  \sqrt{q}\log q+ h_2(E, q, \gamma, \epsilon, m)\sqrt{q}~~\text{if}~~\chi(-1)=-1.
  \end{dcases}
\end{align*}
\end{theorem}
Upper bounds for $ h_{1/2}(E, q, \gamma, \epsilon)$ appear in the following tables. We first fix small $\epsilon$ and $q\ge q_0$ and will give $ h_{1/2}(E, q, \gamma, \epsilon)$ minimized in $\gamma$ and $E$.

\begin{center}
\bgroup
\def\arraystretch{1.25}%
\captionof{table}{Small $\epsilon$} \label{tab:1} 
\begin{tabular}{| c | c | c | c |}
  \hline   
  $\epsilon$ & $\log \log q_0$ & $h_1(E, q, \gamma, \epsilon, m)$ & $h_2(E, q, \gamma, \epsilon)$\\[0.5ex] 
  \hline  \hline
    $\frac{1}{10}$& 22  & 2727 & 5449\\
  \hline
   $\frac{1}{100}$& 209  & 3939 &7872\\
  \hline
  $\frac{1}{1000}$&2081  & 4092 &8180\\
  \hline
  $\frac{1}{10000}$&20800  & 4108 &8210\\
  \hline
\end{tabular}
\egroup
\end{center}

In the following table we will work with $\epsilon$ near $1/8$ to minimize $q$.

\begin{center}
\bgroup
\def\arraystretch{1.25}%
\captionof{table}{Large $\epsilon$} \label{tab:2} 
\begin{tabular}{| c | c | c | c |}
  \hline
  $\epsilon$ & $\log \log q_0$  & $h_1(E, q, \gamma, \epsilon, m)$ & $h_2(E, q, \gamma, \epsilon)$\\[0.5ex] 
  \hline \hline
  $\frac{1}{8}(1-\frac{1}{10})$&19.7   & 2594 &5183\\
  \hline
  $\frac{1}{8}(1-\frac{1}{100})$& 17.99 & 2480 &4955\\
  \hline  
  $\frac{1}{8}(1-\frac{1}{1000})$&17.89  & 2469 &4933\\
  \hline
  $\frac{1}{8}(1-\frac{1}{10000})$&17.83   & 2468 &4931\\
  \hline
\end{tabular}
\egroup
\end{center}
See Table \ref{tab:4} to see the ranges in which the above results are better than those in \cite{F-S}. 
We will also prove a version of the above two tables for characters with moduli with $d(q)$ fixed, an interesting case is certainly when $d(q)=2$ and the modulus is thus prime. See Table \ref{tab:4} in Section \ref{OP}.
Theorem \ref{EP-V} depends on the following result.
\begin{theorem}
\label{P-V}
Take any $h$,~$m$ and $q$ as in Theorem \ref{TB}. With $\chi$ a primitive Dirichlet character of modulo $q$, fixed $\gamma$, if $q$ is such that $h(\log q)^{\gamma}< q^{\frac{1}{4}}$, with $E\ge 4$ and $C$ the Euler--Mascheroni constant, we have the following result.
\begin{align*}
  \left|S(N, \chi)
   \right| \le 
  \begin{dcases}
     \frac{2}{\pi^2}\Big(\frac{3}{8}+\epsilon\Big)\sqrt{q}\log q +\left(\frac{2n(q,\epsilon)}{ \pi^2} +j\right)\sqrt{q}~\text{if}~\chi(-1)=1,\\
    ~\\
   \frac{1}{\pi} \Big(\frac{3}{8}+\epsilon\Big)  \sqrt{q}\log q+ \left(\frac{n(q,\epsilon)}{\pi}+j\right)\sqrt{q}~\text{if}~\chi(-1)=-1,
  \end{dcases}
\end{align*}
with
\begin{align*}
&j=j(E, q, \gamma, \epsilon,m)=\frac{c(\chi)}{\pi}\Big(\frac{1}{\log q}+\frac{5}{8}-\epsilon\Big)c(E, q, \gamma, \epsilon,m)+1+\frac{\left(e^{\pi}-1-\pi\right)}{2\pi} ,\\
  &c(\chi)=
  \begin{dcases}
     1~~\text{if}~~\chi(-1)=1,\\
    ~\\
   2~~\text{if}~~\chi(-1)=-1,
  \end{dcases}\\
&n(q,\epsilon)= C +\log 2 +\frac{3}{(\frac{1}{4}+\epsilon) q},
 \end{align*}
 \begin{align*}
&c(E, q, \gamma, \epsilon, m)=\max\Big\{(\frac{3}{8}+\epsilon)^{-1}c_1(1,E,\log^{\gamma} q)   + \frac{c_2(1,E, \log^{\gamma} q)  \log^{\frac{3}{2}} (E \log^{\gamma} q)}{ \log^{\frac{\gamma}{2}-1} q } ,\\
 &~~~~~~~~~~~~~~~3 m\log^{2\gamma+1} q \left(1 + 4 \pi \log^{\gamma} q\right)\frac{(\log q \log \log q)^{\frac{1}{2}}}{q^{\epsilon/2-\frac{3 \log 2 }{2\log \log q}\left(1+\frac{1}{\log \log q}+\frac{4.7626}{(\log \log q)^2} \right)}} \Big\},\\
&c_1(B,E,R)=(1+2\pi)b_1(B,E)+B2\pi\frac{2\log R}{R^2},\\
&c_2(B,E, R)= (1+2\pi)\Big( b_2(B,E)\frac{(e^C \log \log R+ \frac{2.51}{\log \log R})^{\frac{1}{2}}}{(\log ER)^{\frac{3}{2}}}+b_3(B,E,R,q)\Big),\\
&b_1(B,E)=B+9.82B^2+ 8.12(a_1+a_4) E+ 18.9a_3+3.46a_1+4.06a_4,\\
&b_2(B,E)=a_2\frac{\sqrt{2}E+1}{\log 2},\\
&b_3(B,E, R, q)=B^2 \frac{7.63}{\sqrt{E}}\frac{\log (4ER)}{\log (ER)^{\frac{3}{2}}}\left(1+\frac{\log(64/E)+1}{\log q} \right) +a_5 1.48+\\&~~~~~~~+1.48 \frac{a_6}{\sqrt{\log(ER)}}+ \frac{a_5}{\log 2}\frac{\log(2R)}{\log(ER))^{\frac{3}{2}}}+\frac{a_6}{\log 2}\frac{1}{(\log(ER))^{\frac{3}{2}}} ,\\\
&a_1= B^21.59, ~~a_2=B^2z\frac{\pi^2}{6}, ~~a_3=a_5=B^2z 0.96, \\
&a_4=B^2 z,~~a_6=B^2 1.32 z,~~ z=8\sqrt{2\prod _{p>2} \left( 1+\frac{1}{p^3-p^2-2p}\right)}.
\end{align*}

\end{theorem}
We will refer to the above defined functions through the paper.
The outline of this article is as follows. In Section \ref{SM-V} we prove Corollary \ref{c1} and in Section \ref{B} we prove Theorem \ref{TB}. We proceed using these two results in Section \ref{PVS} to prove Lemma \ref{LI1} and Theorem \ref{P-V}. We conclude proving Theorem \ref{EP-V} in Section \ref{OP}.
\section{Explicit Montgomery--Vaughan result}
\label{SM-V}
We aim to prove the following explicit result following \citep{M-V}.
\begin{theorem}
\label{t1}
Suppose that $4\le  q \le N$, $E\ge 4$ and $(a,q)=1$. Then
\begin{equation*}
 S(a/q) \le b_1(B,E)\frac{N}{\log N}+b_2(B,E)\frac{N}{\phi(q)^{\frac{1}{2}}}
+b_3(B,E, N/q,q)\sqrt{Nq}\log^{\frac{3}{2}} EN/q,
 \end{equation*}
uniformly for $f\in F$.
\end{theorem}
We will deduce Corollary \ref{c1} from Theorem \ref{t1}.\\
An essential theorem to make the Montgomery--Vaughan result explicit is the Brun--Titchmarsh inequality \citep[Theorem 3.7]{MV1}.
\begin{theorem}
\label{B-T}
Let $a$ and $q$ be coprime integers, and let $x$ and $y$ be real numbers with $1\le q<y \le x$. Then we have
\begin{equation*}
\pi (x+y; q, a)- \pi (x; q, a) \le \frac{2y}{\varphi(q) \log \frac{y}{q}},
\end{equation*}
for all $q \le x$.
\end{theorem}
We introduce a precise enough result on primes from \cite{Rosser}.
\begin{theorem}
\label{R1}
For $x>1$ we have
\begin{equation*}
\pi(x)< 1.25506 \frac{x}{\log x}.
\end{equation*}
\end{theorem}
We now introduce a result on the logarithm integral.
\begin{lemma}
\label{Li}
For $x>2$ we have
\begin{equation*}
\mathrm{Li}(x):=\int_2^x \frac{1}{\log t} dt\le 1.37\frac{x}{\log x}.
\end{equation*}
\end{lemma}
\begin{proof}
By Lemma 4.10 in \cite{Bennett} we have that
\begin{equation*}
\mathrm{Li}(x):=\int_2^x \frac{1}{\log t} dt\le 1.2\frac{x}{\log x}~~\text{for}~~x\ge 1865.
\end{equation*}
The result then follows computing $\mathrm{Li}(x)\frac{\log x}{x}$ for $2< x< 1865$.
\end{proof}
Note that for our applications the above results are sharp enough.\\
We now introduce a result by Siebert \citep{Siebert}.
\begin{theorem}
\label{thm:siebert}
Let $a \neq 0$, $b \neq 0$ be integers with $(a,b)=1$, $2\nmid a, b$. Then we have for $x > 1$
\begin{equation}
\label{sieb}
\sum_{p \le x,~ap+b \in P} 1 \le 16  \prod_{p>2} \left(1-\frac{1}{p^2}\right) \prod _{p|ab,~p>2}\frac{p-1}{p-2}\frac{x}{\log^2 x},
\end{equation}
where $P$ denotes the set of all prime numbers.
\end{theorem}
Note that an improvement on the leading constant in \eqref{sieb} would lead to a significant improvement in the final result.
We now introduce some elementary results. The following upper bounds are obtained by splitting the sum in two parts, estimating the first with computer aid and the second simply by integration.
\begin{lemma}
\label{dis}
\begin{equation*}
\sum_1^{\infty} \frac{1}{2^{\frac{n}{2}}}=\frac{\sqrt{2}}{\sqrt{2}-1},
~~~~~
 \sum_1^{\lfloor \log_2 N \rfloor} 2^{\frac{n}{2}}\le \frac{\sqrt{N}-1}{\sqrt{2}-1},
~~~~~
 \sum_1^{\infty} \frac{\sqrt{n}}{2^{\frac{n}{2}}}\le 4.15,
\end{equation*}
\begin{equation*}
 \sum_1^{\infty} \frac{\sqrt{n+1}}{2^{\frac{n}{2}}}\le  4.87,
~~~~~
\sum_{p \ge 2} \frac{ (\log p)^2}{p^2} \le 0.71,
~~~~~
\sum_{p \ge 2} \frac{ \log p}{(p-1)^2} \le 1.27,
\end{equation*}
\begin{equation*}
\sum_{p \ge 2} \frac{ \log p}{p(p-1)} \le 0.8,
~~~~~
\prod _{p>2} \left( 1+\frac{1}{p^3-p^2-2p}\right)\le e^{0.1},
\end{equation*}
\begin{equation*}
\sum_{p,j\ge 2}\frac{j\log p}{p^j}\le0.96,
~~~~~
\sum_{n \ge 1} \frac{ \log n}{n^2} \le 0.94 .
\end{equation*}
\end{lemma}
\subsection{Reduction to bilinear forms}
Note in the following that in the applications we will take $B=1$.
\begin{lemma}
\label{lem:bilinear}
Let $f$ be a multiplicative function satisfying~\eqref{2} and $g$ be any real valued function. Then for any integer $N$ we have 
\begin{equation*}
\left|\sum_{1\le n \le N}f(n)e(g(n))\right|\le (B+2.56B^2) \frac{N}{\log{N}} 
\end{equation*}
\begin{equation*}
\label{lem:eq}
+\frac{1}{\log{N}}\left|\sum_{1\le np \le N} f(n)f(p)(\log{p}) e(g(np))\right|.
\end{equation*}
\end{lemma}
\begin{proof}
We first note that, from \eqref{2},
\begin{equation*}
\left|\sum_{1\le n \le N} f(n)\log (N/n) e(g(n))\right|  \le BN,
\end{equation*}
and hence 
\begin{align}
\label{eq:Spart1}
\left|\sum_{1\le n \le N}f(n)e(g(n))\right|\le \frac{BN}{\log{N}}+\frac{1}{\log{N}}\left|\sum_{1\le n \le N} f(n)(\log{n}) e(g(n))\right|.
\end{align}
Since $\log{n}=\sum_{m|n}\Lambda(m)$ we have 
\begin{align}
\label{eq:fnfnm}
\sum_{1\le n \le N} f(n)(\log{n}) e(g(n))=\sum_{1\le mn \le N} f(mn)\Lambda(m) e(g(mn)).
\end{align}
Our next step is to replace $f(mn)$ with $f(m)f(n)$ and thus we bound
\begin{equation}
\label{eq:Tdef}
T = \sum_{nm \le N} \Lambda(m) |f(mn)-f(m)f(n)| \le \Sigma_1 + \Sigma_2,
\end{equation}
where
\begin{equation*}
\Sigma_1= \sum_{p,k \ge 1} \sum_{\substack{n \le Np^{-k} \\ p|n}} (\log p) |f(p^kn)|.
\end{equation*}
and
\begin{equation*}
\Sigma_2= \sum_{p,k \ge 1}(\log p)|f(p^k)|\sum_{j \ge 1} |f(p^j)|\sum_{m \le Np^{-k-j}}|f(m)|.
\end{equation*}
Collecting together those terms in $\Sigma_1$ such that $p^k n$ is exactly divisible by $p^j$ and by partial summation, using \eqref{2} and Lemma \ref{dis}, we obtain
\begin{equation*}
\Sigma_1 \le \sum_{p,j \ge 2}(\log p)|f(p^j)|(j-1)\sum_{m \le Np^{-j}} |f(m)|\le B^2 N\sum_{p,j \ge 2} j p^{-j}\log p \le B^2 0.96 N.
\end{equation*} 
By \eqref{2},
\begin{equation*}
\Sigma_2 \le B^2 N \sum_{p,j,k \ge 1} p^{-j-k} \log p = B^2 N \sum_{p \ge 2} \log p \left( \sum_{j\ge 1}p^{-j}\right)^2,
\end{equation*}
thus
\begin{equation*}
\Sigma_2 \le B^2 N \sum_{p \ge 2} \frac{ \log p}{(p-1)^2} 
\end{equation*}
and, using Lemma \ref{dis}, 
\begin{equation}
\label{sigma2}
\Sigma_2 \le 0.8 B^2 N   .
\end{equation}
Thus
\begin{equation*}
T \le B^2 1.76 N,
\end{equation*}
and hence by~\eqref{eq:Spart1},~\eqref{eq:fnfnm} and~\eqref{eq:Tdef}
\begin{align}
\label{eq:fnfnm1}
\nonumber \left|\sum_{1\le n \le N}f(n)e(g(n))\right|\le & (B+1.76B^2) \frac{ N}{\log{N}} \\ &+\frac{1}{\log{N}}\left|\sum_{1\le nm \le N} f(n)f(m)\Lambda(m) e(g(mn))\right|.
\end{align}
Those pairs $m$, $n$  in which $m$ is of the form $p^k$ with $k\ge 2$ contribute an amount to the sum which is bounded by 
\begin{equation*}
\sum_{p,k\ge 2}|f(p^k)|(\log p)\sum_{n \le Np^{-k}} |f(n)|.
\end{equation*}
By \eqref{2} this is
\begin{equation*}
\le B^2 N \sum_{p,k\ge 2}p^{-k} \log p = B^2 N \sum_{p\ge 2}\log p \sum_{k\ge 2}p^{-k} \le B^2 N \sum_{p\ge 2}\frac{\log p }{p(p-1)}.
\end{equation*}

Now, using Lemma \eqref{dis}, we have that 
\begin{equation*}
 \sum_{p,k\ge 2}|f(p^k)|(\log p)\sum_{n \le Np^{-k}} |f(n)|\le 0.8 B^2 N,
\end{equation*}
and hence by~\eqref{eq:fnfnm1} the proof is completed.
\end{proof}
\subsection{Partition of hyperbola into rectangles}
We now partition the summation, over the domain $1\le pn\le N$, occurring in Lemma \ref{lem:bilinear} into rectangles and their complements. Assume $N\ge q$ and let
\begin{equation}
\label{J}
\mathbf{J}_i=\min\{i+1,[\log_2 N]-i+1, [\frac{1}{2}\log_2(EN/q)]\},
\end{equation}
with $E\ge 4$.
Define
\begin{equation*}
\mathbf{R}_i=(0,2^i]\times (N2^{-i-1},N 2^{-i}], \quad(0\le i \le \log_2 N).
\end{equation*}
In the remaining regions we place additional rectangles $\mathbf{R}_{ijk}$, for $j=1,2,...,\mathbf{J}_i$ and for each $j$,~ $2^{j-1} < k \le 2^j$, defined as
\begin{equation}
\label{R}
\mathbf{R}_{ijk}=(2^{i+j}/k,2^{i+j+1}/(2k-1)]\times ((k-1)N2^{-i-j},(2k-1)N2^{-i-j-1}].
\end{equation}
We do this for $j=1,2,...,\mathbf{J}_i$. The choice of $\mathbf{J}_i$ ensures that each $\mathbf{R}_{ijk}$ is a rectangle of the form $(P',P'']\times (N',N'']$
with
\begin{equation*}
P''-P' \ge \frac{1}{4},~ \space N''-N' \ge \frac{1}{4},~ \space (P''-P')(N''-N')\gg q.
\end{equation*}
Let $\mathbf{E}$ denote the set of points $(p,n)$ with $pn \le N$ which do not lie in any rectangle $\mathbf{R}_i$ or $\mathbf{R}_{ijk}$. Then $\mathbf{E}$ is the union of $\mathbf{E}_1$, $\mathbf{E}_2$ and $\mathbf{E}_3$, the unions of those $\mathbf{H}_i=\{(p,n)\in \mathbf{E}:(1,n) \in \mathbf{R}_i\}$ with $\mathbf{J}_i =i+1$, $\mathbf{J}_i=[\log_2 N]-i+1$ and $\mathbf{J}_i=[\frac{1}{2}\log_2(E N/q)]$ respectively.
\begin{lemma}
 The following estimate for the sum on the right of \eqref{lem:eq}, from the points $(p,n)$ in $\mathbf{E}$, holds
\begin{align}
\label{ext1}
\nonumber\sum_{\mathbf{E}} f(p)f(n)& e(pna/q) \log q \le B^27.26N +\\
 & +B^2 \frac{7.63}{\sqrt{E}} (Nq)^{\frac{1}{2}} (\log(4EN/q))^{\frac{1}{2}}(1+\log(64/E)+\log q).
\end{align}
\end{lemma}
\begin{proof}
Consider $\mathbf{E}_1$. For a given $p$, the number of $n$ for which $(p,n)\in \mathbf{E}_1$ is $\le 4 N p^{-2} $ for $p\le 2\sqrt{N}$, and this holds for any $(p,n)\in \mathbf{E}_1$, and for a given $n$, there are $\le 2$ primes $p$ for which $(p,n)\in \mathbf{E}_1$. Hence, by Cauchy's inequality,
\begin{align*}
\sum_{\mathbf{E}_1} | f(p)f(n)| \log p &\le \left( \sum_{\mathbf{E}_1} | f(n)|^2\right)^{\frac{1}{2}}\left( \sum_{\mathbf{E}_1} B^2(\log p)^2\right)^{\frac{1}{2}}\le \\
&\le B \left( 2 \sum_{n \le N} | f(n)|^2\right)^{\frac{1}{2}}\left( \sum_{p \le 2\sqrt{N}} 4 N p^{-2}  (\log p)^2\right)^{\frac{1}{2}},
\end{align*}
which, using Lemma \ref{dis}, is bounded above by $2.39  B^2 N$.
For each pair $(p, n) \in \mathbf{E}_2$ we see that $n \le (2N)^{\frac{1}{2}}$, and for a given $n$ the $p$ with $(p, n) \in \mathbf{E}_2$ all lie in an interval of length $4Nn^{-2}$. Thus by Theorem \ref{B-T} there are 
\begin{equation*}
\le 8\frac{N}{n^2\log 4Nn^{-2}}
\end{equation*}
such $p$. For a given $p$ there is at most one $n$ for which $(p,n)\in \mathbf{E}_2$. 
We have by partial summation
\begin{equation*}
\sum_{n \le \sqrt{2N}} |f(n)|^2 \frac{N}{n^2 \log 4Nn^{-2}} \le B^2 \left( \sum_{n \le \sqrt{N}} \frac{N}{n^2 \log 4Nn^{-2}} +\frac{(\sqrt{2}-1)\sqrt{N}}{\log 2}\right)
\end{equation*} 
\begin{equation*}
\le B^2 N \left (\frac{1 }{\sqrt{2N} \log 2} - \int_1^{\sqrt{N}} \frac{2}{x^2 \log 4Nx^{-2}}\left( \frac{1}{\log 4Nx^{-2}}-1\right)dx+\frac{(\sqrt{2}-1)}{\sqrt{N}\log 2 }\right)
\end{equation*}
\begin{equation*}
\le N B^2\big(\frac{\sqrt{2}-1}{2}\frac{1 }{\sqrt{N} \log 2} +\frac{1}{ \log 4N}+\frac{\text{Li}(\sqrt{N})}{4 \sqrt{N}} +\frac{(\sqrt{2}-1)}{\sqrt{N}\log 2 } \big),
\end{equation*}
which, by Lemma \ref{Li}, is bounded above by $ 2.35\frac{B^2N }{ \log N}$
 and 
 \begin{equation*}
\sum_{p \le N} \log^2 p \le \pi (N) \log^2 N - 2\int_2^N \pi(x) \frac{2 \log x}{x} dx \le 1.26 N \log N.
\end{equation*}
Thus
\begin{equation*}
\sum_{\mathbf{E}_2} | f(p)f(n)| \log p \le \left( \sum_{\mathbf{E}_2} | f(n)|^2\right)^{\frac{1}{2}}\left( \sum_{\mathbf{E}_2}B^2 (\log p)^2\right)^{\frac{1}{2}}\le 
\end{equation*}
\begin{equation*}
\le B\left(  \sum_{n \le \sqrt{2N}} | f(n)|^2 8\frac{N}{n^2\log 4Nn^{-2}}\right)^{\frac{1}{2}}\left( \sum_{p \le N}  (\log p)^2\right)^{\frac{1}{2}}\le 4.87 B^2 N.
\end{equation*}
For each such $p$ the number of $n$ for which $(p, n) \in \mathbf{E}_3$ is $\le \frac{8}{\sqrt{E}}(Nq)^{\frac{1}{2}} p^{-1} $. For each $n$, the $p$ for which $(p,n)  \in \mathbf{E}_3$ lie in an interval of length $\le \frac{8}{\sqrt{E}}(Nq)^{\frac{1}{2}} n^{-1}$, so that, by Theorem \ref{B-T}, there are 
\begin{equation*}
\le \frac{16}{\sqrt{E}} \frac{\sqrt{Nq}}{n\log \frac{64}{E}Nqn^{-2}} 
\end{equation*}
such $p$. 
When $(p, n) \in \mathbf{E}_3$ we have $\frac{\sqrt{E}}{4}(N/q)^{\frac{1}{2}} \le p \le \frac{8}{\sqrt{E}}(Nq)^{\frac{1}{2}}$ and $\frac{\sqrt{E}}{16}(N/q)^{\frac{1}{2}} \le n \le \frac{4}{\sqrt{E}}(Nq)^{\frac{1}{2}}$, thus the following sum on $p$ and $n$ will be restricted to these intervals.
Using Theorem \ref{R1} we have
\begin{equation*}
\sum_p \frac{\log p}{p}\le 1.26 (1+\log(32/E)+\log q)
\end{equation*}
and
\begin{equation*}
\sum_n |f(n)^2| \frac{1}{n} \le B^2 \left( 1 +\log(64/E)+\log q \right).
\end{equation*}
Therefore
\begin{equation*}
\sum_{\mathbf{E}_3} | f(p)f(n)| \log p \le \left( \sum_{\mathbf{E}_3} | f(n)|^2 \log N/n\right)^{\frac{1}{2}}\left( B^2\sum_{\mathbf{E}_3} \log p\right)^{\frac{1}{2}} 
\end{equation*}
\begin{equation*}
\le B \sqrt{Nq}\left( \frac{16}{\sqrt{E}}\sum_{n } | f(n)|^2 \frac{\log N/n }{n\log \frac{64}{E}Nqn^{-2}}\right)^{\frac{1}{2}}\left( \frac{8}{\sqrt{E}}\sum_{p }  \frac{\log p}{p}\right)^{\frac{1}{2}},
\end{equation*}
and as the above ratio of the logarithms is less than $\frac{1}{4\log 2}\log(4 EN/q)$ 
\begin{align*}
\sum_{\mathbf{E}_3} | f(p)f(n)| \log p& \le \\&\le B^2 (Nq)^{\frac{1}{2}}  \frac{7.63}{\sqrt{E}}(\log(4 E N/q))^{\frac{1}{2}} (1 +\log(64/E)+\log q).
\end{align*}
Combining the above estimates gives \eqref{ext1}.

\end{proof}

\subsection{The fundamental estimate}
Here we will develop a tool to bound the bilinear forms onto the rectangles defined in the previous section, in doing this we follow \cite[Section 4]{M-V}.
\begin{lemma}
\label{lem:bilinearrectangles}
Let $K,Q,M,X,Y\in \mathbb{R}^+$ and for each $1\le k \le K$ let 
\begin{align*}
\cR(k)=\cL(k)\times \cM(k),
\end{align*}
be a rectangle satisfying
\begin{align}
\label{eq:LMk}
\cL(k)\subseteq (0,Q), \quad \cM(k)\subseteq (0,M],
\end{align}
and
\begin{align*}
\cL(k)=(Q'(k),Q''(k)], \quad \cM(k)=(M'(k),M''(k)],
\end{align*}
for some $Q'(k),Q''(k)$ and $M'(k),M''(k)$ satisfying
\begin{align*}
Q''(k)-Q'(k)\le X, \quad M''(k)-M'(k)\le Y, \quad M''(k)\le 2M'(k).
\end{align*}
Suppose that the rectangles $\cR(k)$ are disjoint for each $1\le k \le K$. Then for any function $f(n)$ satisfying \eqref{2},  define 
\begin{equation*}
I=\sum_{k=1}^{K}\sum_{(p,n)\in \cR(k)}f(p)f(n)e(pna/q)\log{p}.
\end{equation*}
Then, if $(a,q)=1$ and $q \le XY$, we have 
\begin{align}
\label{eqI}
\nonumber I \le  B^2 \Big( 2.52MQ(Y+1)\log  Q+128 \prod _{p>2} \left( 1+\frac{1}{p^3-p^2-2p}\right)MQ \\ \Big(\frac{\pi^4}{6^2}\frac{XY}{\varphi (q)}
  +0.91X+ Y\log2X+0.91 q\Big(\log(2XY/q)+ 0.94\Big)+\frac{q\pi^2}{12}\Big)\Big)^{\frac{1}{2}}.
\end{align}
\end{lemma}
\begin{proof}
Let $\cR=\cL\times \cM$ be one of the rectangles $\cR(k)$. By Cauchy's inequality
\begin{align}
\label{FE2}
\nonumber \left|\sum_{(p,n)\in \cR}f(p)f(n)e(pna/q)\log{p}\right| \le & \left(\sum_{n \in \cM} |f(n)|^2 \right)\\
&\cdot \left(\sum_{n \in \cM}\left|\sum_{p\in \cL}f(p)f(n)e(pna/q)\log{p}\right|^2 \right).
\end{align}
We now introduce the smoothing factor
\begin{equation*}
w(n)=\max\{0,2-|2n-2M'-Y|Y^{-1}\},
\end{equation*}
such that $w(n) \ge 1$ for $n \in \cM$. Note that the above is a variation of Fejer kernel and we choose it following Montgomery and Vaughan. Aiming to improve the result it would surely be interested to chose other kernels. We also introduce $g(n)=\max\{0,1-|n|\}$ and note that for the Fourier transform of $g$ we have  
\begin{equation}
\label{eq:g}
|\widehat{g}(n)|=\left(\frac{\sin \pi n}{\pi n}\right)^2.
\end{equation}
Thus the second factor on the right of \eqref{FE2} is bounded above by 
\begin{equation*}
 \sum_n w(n)\Big|\sum_{p\in \cL}f(p)f(n)e(pna/q)\log{p}\Big|^2  =
\end{equation*}
\begin{equation*}
=\sum_{p,p' \in \cL} f(p)f(p') (\log p)(\log p')\sum_n w(n)e((p-p')na/q), 
\end{equation*}
using Poisson formula and \eqref{eq:g}, we obtain
\begin{equation*}
\le B^2 (\log Q)^2 \sum_{p,p' \in \cL} \min \left\{2Y, \frac{\frac{4}{\pi^2}+\frac{1}{6}}{ Y \Vert \frac{(p-p')a}{q}\Vert^2 }\right\}.
\end{equation*}
By Cauchy's inequality, and Theorem \ref{R1},
\begin{align*}
I \le &B (\log Q)  \left(\sum_k \sum_{n \in \cM} |f(n)|^2 \right)^{\frac{1}{2}}\left(\sum_k  \sum_{p,p' \in \cL} \min \left\{2Y, \frac{0.58}{ Y \Vert \frac{(p-p')a}{q}\Vert^2}\right\}\right)^{\frac{1}{2}}   \\
\le& B^2 (\log Q) \sqrt{M} \left( 2.52(Y+1) \frac{Q}{\log Q}+\sum_{0<h \le X}\sum_{\substack{p \le Q \\ p+h=p'}}  \min \left\{2Y, \frac{0.58}{ Y \Vert \frac{ha}{q}\Vert^2}\right\}\right)^{\frac{1}{2}}.
\end{align*}
Now from Theorem \ref{thm:siebert} we obtain
\begin{equation}
\label{I_1}
I \le B^2\left(MQ2.52(Y+1)\log Q+16 \prod _{p>2} \left( 1-\frac{1}{p^2}\right)MQV\right)^{\frac{1}{2}},
\end{equation}
where
\begin{equation*}
V=\sum_{0<h \le X}\prod_{p|h, p>2} \frac{p-1}{p-2} \min \left\{2Y, \frac{0.58}{ Y \Vert \frac{ha}{q}\Vert^2}\right\}.
\end{equation*}
Hence we need to bound $V$. Now we have
\begin{equation*}
\prod_{p|h, p>2} \frac{p-1}{p-2}= \prod_{p|h, p>2} \Big( 1+\frac{2}{p^2-p-2}\Big)\prod_{p|h, p>2} \Big( 1+\frac{1}{p}\Big)   
\end{equation*}
\begin{equation*}
\le \prod_{ p>2} \Big( 1+\frac{2}{p^2-p-2}\Big) \sum_{m|h}\frac{1}{m},
\end{equation*}
so that
\begin{equation*}
V\le \prod_{ p>2} \Big( 1+\frac{2}{p^2-p-2}\Big)\left(\sum_{m \le X} 1/m\sum_{n \le X/m}\min \left\{2Y, \frac{0.58}{ Y \Vert \frac{mna}{q}\Vert^2}\right\}\right).
\end{equation*}
The innermost sum is of the form 
\begin{equation*}
W \le 0.58\frac{1}{Y} \sum_{ b\le Z} \min \left\{(0.58)^{-1}2Y^2,\frac{1}{ \Vert \frac{ba}{r}\Vert^2}\right\}
\end{equation*}
with $r=q/(m,q)$ and $(b,r)=1$. Using \cite[Lemma 14]{Korobov} this is seen to satisfy
\begin{equation*}
W =  \min\left\{2YZ, 4\sqrt{2}\left(0.58\right)^{\frac{1}{2}}(Z+r)\left(\left(\frac{2}{0.58}\right)^{\frac{1}{2}} Y+r\right)r^{-1}\right\}.
\end{equation*}
Therefore by
\begin{align*}
\sum_{\substack{m \le X \\ (m,q)XY \le mq}}\frac{2XY}{m^2}+
 \sum_{\substack{m \le X \\ (m,q)XY > mq}}\frac{1}{m}\Big( \frac{8XY}{mq}(m,q)+4\sqrt{2}\left(0.58\right)^{\frac{1}{2}}\frac{X}{m}+\\+8Y+4\sqrt{2}\left(0.58\right)^{\frac{1}{2}}\frac{q}{(m,q)}\Big)\le
\end{align*}
 \begin{align*}
\le 2\sum_{r|q} \sum_{s>XY/q} \frac{XY}{r^2s^2}+8 \sum_{r|q}\sum_s \frac{XY}{rs^2q}+ \frac{\pi^2}{6}4\sqrt{2}\left(0.58\right)^{\frac{1}{2}} X+8 Y\log2X+\\+4\sqrt{2}\left(0.58\right)^{\frac{1}{2}} \sum_{r|q}\sum_{s<XYr/q}\frac{q}{r^2s},
\end{align*}
we obtain
\begin{align*}
V\le &  8\prod_{ p>2} \Big( 1+\frac{2}{p^2-p-2}\Big)\Big(\frac{\pi^4}{6^2} XY\varphi (q)^{-1} + 0.91X+\\
&+Y\log2X+ 0.91 q\Big(\log(2XY/q)+ \sum_{n=1}^{\infty} \frac{\log n}{n^2}\Big)+\frac{q\pi^2}{12}\Big).
\end{align*}
Thus from the above bound, Lemma \ref{dis} and \eqref{I_1} we obtain the desired result.
\end{proof}
\subsection{Completion of the Proof of Theorem \ref{t1}}
Note that in the following argument we will extensively use Lemma \ref{dis} and refer to the notation of Theorem \ref{t1}.
We first apply \eqref{eqI} to the rectangle $\mathbf{R}_i$. We take $K=1$, $X=Q=2^i$, $Y=M=N2^{-i}$. Thus 
\begin{align}
 \label{Ri}
\nonumber \mid \sum_{(p,n) \in \mathbf{R}_i} &f(p)f(n) e(pna/q) \log q  \mid \le 
 \sqrt{\log 2}a_1 N\sqrt{\frac{i}{2^i}}+a_2\frac{N}{\sqrt{\phi(q)}}+\\ &+a_3\sqrt{N2^i}+\sqrt{\log 2}a_4 N\sqrt{\frac{i+1}{2^i}}+a_5\sqrt{Nq \log (2N/q)}+a_6\sqrt{qN}.
\end{align}
Next, for each pair $i$, $j$ with $1 \le j \le J_i$ we apply \eqref{eqI} to the family of $ 2^{j-1}$ rectangles $\mathbf{R}_{ijk}$ with $2^{j-1} < k \le 2^j$. By \eqref{R} we may take $K =2^{j-1}$, $M=N2^{-i}$, $Q=2^{i+1}$, $X=2^{i-j+1}$, $Y=\frac{E}{2} N2^{-i-j}$. Thus, by \eqref{J}, $XY \ge q$, so that the conditions for \eqref{eqI} to hold are satisfied. Hence
\begin{equation*}
\left|\sum_{2^{j-1} < k \le 2^j} \sum_{(p,n) \in \mathbf{R}_{ijk}} f(p)f(n) e(pna/q) \log q  \right| \le \sqrt{\log 2}(a_1+a_4)_1 E\sqrt{\frac{i+1}{2^{i+j}}}N+
\end{equation*}
\begin{equation*}
+\sqrt{2}a_2 E\frac{1}{2^j}\frac{N}{\sqrt{\phi(q)}}+a_3 2 \sqrt{2^{i-j}}\sqrt{N}+ a_5 (2Nq \log \frac{E}{2} N/q))^{\frac{1}{2}}+a_6 \sqrt{2Nq}.
\end{equation*}
By \eqref{J} $J_i \le  \frac{1}{2}\log_2 (EN/q)$. Hence, summing over those $j$ with  $1 \le j \le J_i$ we obtain
\begin{equation*}
\left| \sum_{1 \le j \le J_i}\sum_{2^{j-1} < k \le 2^j}\sum_{(p,n) \in \mathbf{R}_{ijk}}  f(p)f(n) e(pna/q) \log q  \right|
\end{equation*}
\begin{equation*}
\le \sqrt{\log 2}\frac{\sqrt{2}}{\sqrt{2}-1}(a_1+a_4)_1 E\sqrt{\frac{i+1}{2^{i}}} N+\sqrt{2}a_2 E\frac{N}{\sqrt{\phi(q)}}+\frac{2\sqrt{2}}{\sqrt{2}-1}a_3  \sqrt{2^{i}}\sqrt{N}+
\end{equation*}
\begin{equation*}
+\frac{\sqrt{2}}{2} \left(a_5 (Nq \log (\frac{E}{2} N/q))^{\frac{1}{2}} +a_6 \sqrt{Nq}\right)\log_2 (EN/q).
\end{equation*}
Therefore, by \eqref{Ri}, summing over $i$ with $0 \le i \le \log_2 N$, we can obtain  
\begin{equation*}
\left| \sum_{\substack{pn \le N  \\  (p,n)   \not\in \mathbf{E}  }} f(p)f(n) e(pna/q) \log q \right| \le 
\end{equation*}
\begin{equation*}\Big(  8.12(a_1+a_4) E+ 18.9a_3+3.46a_1+4.6 a_4\Big)N+a_2 \frac{E\sqrt{2}+1}{\log 2}\frac{N\log N}{\phi(q)^{\frac{1}{2}}}+
\end{equation*}
\begin{equation*}
+\Big(\frac{a_5 }{\sqrt{2}\log 2}+\frac{a_6 }{\sqrt{2}\log 2}\frac{1}{\sqrt{\log(EN/q)}}+\frac{a_5\log(2N/q)}{(\log(EN/q))^{\frac{3}{2}}}+ \frac{a_6}{(\log(EN/q))^{\frac{3}{2}}}\Big)\cdot
 \end{equation*}
\begin{equation*}
\cdot\frac{(Nq )^{\frac{1}{2}}}{\log 2} (\log (EN/q))^{\frac{3}{2}}\log N. 
\end{equation*}

This with \eqref{ext1} and   Lemma \ref{lem:bilinear} gives  Theorem \ref{t1}.
\subsection{Proof of Corollary \ref{c1}}
Let $S(\alpha, u)= \sum_{n \le u} f(n) e(n\alpha)$. Then
\begin{align*}
S(\alpha)=e((\alpha - \beta)N)S(\beta,N)-2\pi i(\alpha-\beta) \int_1^N S(\beta,u)e((\alpha-\beta)u)du
\end{align*}
Suppose that $\beta = b/r$ with $(b,r)=1$ and $r \le N$. Then, on using that $|S(\alpha, u)|\le B r$ when $u \le r$ and Theorem \ref{t1} when $u > r$ we obtain 
\begin{align}
\label{eqq1}
\nonumber \left|S(\alpha) \right|\le &\left( b_1(B,E)\frac{N}{\log N}+\frac{b_2(B,E)N}{\sqrt{\phi(r)}}
+b_3(B,E, N/r))\sqrt{rN}(\log (EN/r))^{\frac{3}{2}}   \right) \\&
\cdot(1+2\pi (N-r)| \alpha -b/r|)+Br^22\pi | \alpha -b/r|.
\end{align}
Here we use, from \cite[Theorem 15]{Rosser}, that for $n\ge 3$
\begin{equation*}
\phi(n) > \frac{n}{e^C \log \log n+ \frac{2.51}{\log \log n}},
\end{equation*}
with $C$ the Euler--Mascheroni constant.
If $q > N^\frac{1}{2}$, then we take $b=a$, $r=q$, which gives 
\begin{equation*}
S(\alpha)\le  (1+2\pi)b_1(B, E) \frac{N}{\log N}+B2\pi+ (1+2\pi)\cdot 
\end{equation*}
\begin{equation*}
\cdot\Big( b_2(B,E)\frac{(e^C \log \log R+ \frac{2.51}{\log \log R})^{\frac{1}{2}}}{(\log ER)^{\frac{3}{2}}}+b_3(B,E, R)\Big)\frac{(\log E R)^{\frac{3}{2}}}{\sqrt{R}}N .
\end{equation*}
If $q \le N^\frac{1}{2}$, then by Dirichlet's theorem there exist $b$, $r$ such that $(b,r)=1$, $r \le 2N/q$ and $\mid \alpha -b/r\mid \le q/(2rN)$. Thus, either $r=q$ or $1 \le | ar -bq|= rq| (\alpha - b/r)-(\alpha-a/q)| \le q^2/(2N)+r/q \le \frac{1}{2}+r/q$, thus  in either case $r \ge \frac{1}{2}q$. Therefore $|\alpha - b/r| \le N^{-1}$ and consequently, by \eqref{eqq1}, Corollary \ref{c1} follows once more.
\section{Explicit Burgess bound for composite moduli}
\label{B}
We now prove Theorem \ref{TB}.
For the following result see~\cite{Robin} and~\cite[p.\ 43]{Sandor}.
\begin{lemma}
\label{lem:tauub}
For any integer $n\ge 3$ we have 
\begin{align*}
\log d (n) \le \frac{\log q }{\log \log q}\left(\log 2+\frac{\log 2}{\log \log q}+\frac{4.7626 \log 2}{(\log \log q)^2} \right).
\end{align*}
\end{lemma}
\begin{theorem}
\label{TB1}
Let $q$,~$k$,~$h$ and $m$ be as in Theorem \ref{P-V}. Let $\chi$ be a primitive character mod $ q$  and $\psi$ be any character mod $k$. For any integers $M$ and $N<q$ we have
\begin{align}
\label{Tr}
\left|\sum_{M<n\le M+N}\psi(n)\chi(n)\right|\le \frac{ mk N^{1/2}q^{3/16}(\log q\log \log q)^{\frac{1}{2}}}{q^{\frac{3 \log 2 }{2\log \log n}\left(1+\frac{1}{\log \log n}+\frac{4.7626}{(\log \log n)^2}\right) }}.
\end{align}

\end{theorem}
The proof of the following is the same as~\cite[Lemma~1]{Trev} which deals with the case $q=p$ prime.
\begin{lemma}
\label{lem:multcong}
For integers $q,M,N,U$ satisfying
\begin{align*}
N<q ~~~~~\text{and}~~~~28\le U\le \frac{N}{12},
\end{align*}
let $I_q(N,U)$ count the number of solutions to the congruence
$$n_1u_1\equiv n_2u_2 \mod{q}, \quad M\le n_1,n_2\le M+N, \quad 1\le u_1,u_2\le U, \quad (u_1u_2,q)=1.$$
We have 
\begin{align*}
I_q(N,U)\le  2UN\left(\frac{NU}{q}+\log(1.85 U) \right).
\end{align*}
\end{lemma}
Using an idea of Burgess~\cite{Burgess2}, with an improvement of Heath-Brown~\cite{HB}, we have the following
\begin{lemma}
\label{lem:burmv}
Let $q,k,V$ be integers with $V<q$. For any primitive $\chi \mod{q}$ we have 
$$\sum_{\lambda=1}^{q}\left|\sum_{v\le V}\chi(\lambda+kv)\right|^4\le 16qk^2V^2+4q^{1/2}k^4V^4d(q)^6.$$
\end{lemma}
\begin{proof}
Let 
$$S=\sum_{\lambda=1}^{q}\left|\sum_{v\le V}\chi(\lambda+kv)\right|^4.$$
Expanding the fourth power and interchanging summation, we have 
\begin{align*}
S \le  \sum_{1\le v_1,\dots,v_4\le kV}\left|\sum_{\lambda=1}^{q}\chi\left(\frac{(\lambda+v_1)(\lambda+v_2)}{(\lambda+v_3)(\lambda+v_4)}\right)\right|.
\end{align*}
Define $A_j=\prod_{i\not= j}(m_j-m_i)$ and $K=(q,A_j)$.
Using \cite[Lemma 7]{Burgess2} and arguing as in Burgess~\cite[Lemma 8]{Burgess2}, we obtain
\begin{equation}
\label{eq:b}
S\le 16qk^2V^2+8^{\tau(q)}q^{1/2}\sum_{j=1}^4\sum'_{m_1,\cdots,m_4} K,
\end{equation}
where $\sum'$ is the sum over all $m_1,\cdots,m_4 \leq kV$, which contains at least $3$ distinct elements and $\tau(n)$ is the function that counts the prime divisors of $n$.
Bounding the right hand side of \eqref{eq:b} as in Heath-Brown~\cite[Lemma~2]{HB}, we get 
\begin{align*}
S\le 16qk^2V^2+4q^{1/2}k^4V^4d(q)^6,
\end{align*}
which completes the proof.
\end{proof}

\subsection{Proof of Theorem~\ref{TB}}
We begin proving the following fundamental result.
\begin{theorem}
\label{Bint}
Let $q$ and $k$ be integers and $g\ge 2$,~$m$ and $h$ positive real numbers. Let $\chi$ be a primitive character modulo $q$  and $\psi$ be any character modulo $k$. Assume that 
\begin{equation*}
q> \max\left\{ \left(\frac{29 g}{m^2\log q \log \log q}\right)^8,\left( \frac{12}{g}\right)^4, k, (hk)^4 \right\},
\end{equation*} 
and
\begin{equation*}
q \ge m^2q^{3/8}\log q \log \log q\ge gq^{1/4}.
\end{equation*}
Define
\begin{align*}
v_1(m,q)=\frac{2(1+\frac{2}{e\log q}) }{m},
\end{align*}
\begin{align*}
v_2(m,q,g)
=\frac{\left(v_1(m,q)\right)^4}{g}+\frac{(\log \log q)^2\log \left(1.85\left(v_1(m,q)\right)^2q^{\frac{3}{8}}\frac{\log q }{g \log \log q}\right)}{(\log q)^2} ,
\end{align*}
and
\begin{align*}
v_3(m,q,g,h)=& 2 g\left(1-\frac{1}{h}-\frac{gq^{\frac{1}{4}}}{m^2q^{3/8}\log q (\log \log q)}\right)^{-1}\left(\frac{17  v_2(m,q,g)}{4g^3}\right)^{\frac{1}{4}}\cdot \\ &\cdot\left( e^{C} +\frac{2.51}{(\log \log q)^2}\right)+\frac{2}{\sqrt{g}}.
\end{align*}
If $v_3(m,q,g,h)\le m$ holds then, for any integers $M,N$, we have
\begin{align}
\label{II}
\left|\sum_{M<n\le M+N}\psi(n)\chi(n)\right|\le mkd(q)^{3/2} N^{1/2}q^{3/16}(\log q)^{\frac{1}{2}} (\log \log q)^{\frac{1}{2}}.
\end{align}

\end{theorem}
\begin{proof}
We proceed by induction on $N$ using \eqref{II}, as for any $K\le m^2q^{3/8}\log q (\log \log q)$ we trivially have 
\begin{align*}
\left|\sum_{M<n\le M+K}\psi(n)\chi(n)\right|\le m kd(q)^{3/2}K^{1/2}q^{3/16}(\log q)^{\frac{1}{2}} (\log \log q)^{\frac{1}{2}}.
\end{align*}
This forms the basis of our induction and we assume \eqref{II} holds for any sum of length strictly less than $N$. Define
\begin{align}
\label{eq:UVdef}
U= \left \lfloor \frac{N}{gq^{1/4}} \right\rfloor, \quad V=\left \lfloor \frac{q^{1/4}}{k} \right\rfloor,
\end{align}
and note that 
\begin{align*}
UV\le \frac{N}{gk}.
\end{align*}
Also note that $\psi \chi$ is a non-principal character, with modulo $\le kq$, for otherwise $\overline{\psi}$ and $\chi$ would be induced by the same primitive character; that is impossible as $\chi$ is primitive modulo $q$ and we have that $q>k$. We thus have by the P\'{o}lya--Vinogradov inequality
\begin{align*}
\left|\sum_{M<n\le M+K}\psi(n)\chi(n)\right|\le 2\sqrt{kq}\log kq,
\end{align*}
and thus for $K>v_1(m,q)^2 q^{\frac{5}{8}}\frac{\log q}{\log \log q}$ Theorem~\ref{TB} holds. Note that using the P\'{o}lya--Vinogradov inequality from \cite{F-S} would allow to improve on $m$, but the above result is good enough for our purposes.
For any integer $y<N$ we have 
\begin{align*}
\sum_{M<n\le M+N}&\psi(n)\chi(n)=\sum_{M-y<n\le M+N-y}\psi(n+y)\chi(n+y) \\
&=\sum_{M<n\le M+N}\psi(n+y)\chi(n+y)+\sum_{M-y<n\le M}\psi(n+y)\chi(n+y) \\ & -\sum_{M+N-y<n\le M+N}\psi(n+y)\chi(n+y),
\end{align*}
and hence
\begin{align*}
\sum_{M<n\le M+N}\psi(n)\chi(n)=\sum_{M<n\le M+N}\psi(n+y)\chi(n+y)+ 2\theta E(y),
\end{align*}
with $E(y)=\max_M\left|\sum_{M<n\le M+y}\psi(n)\chi(n)\right|$ and for some $|\theta|\le 1$  depending on $y$. Let $\cU$ denote the set 
$$\cU=\{ 1\le u \le U \ : \ (u,q)=1 \},$$
and average the above over integers  $kuv$ with $u\in \cU$ and $1\le v \le V$ to get 
\begin{align}
\label{eq:Wbb1}
\left|\sum_{M<n\le M+N}\psi(n)\chi(n)\right|\le \frac{1}{V|\cU|}|W|+ 2\frac{1}{V|\cU|} \sum_{u,v}E(y),
\end{align}
where 
\begin{align*}
W=\sum_{M<n\le M+N}\sum_{u\in \cU}\sum_{1\le v \le V}\psi(n+kuv)\chi(n+kuv).
\end{align*}
For any $u$, $v$  we have $uvk\le N/g$, we thus by the induction hypothesis 
\begin{align}
\label{eq:E}
2\frac{1}{V|\cU|}\sum_{u,v}E(y)&\le \frac{2}{\sqrt{g}}mkd(q)^{3/2}\sqrt{N} q^{3/16}(\log q)^{\frac{1}{2}}(\log \log q)^{\frac{1}{2}}.
\end{align}
Since $\psi$ is a character mod $k$, we have 
\begin{align*}
|W|\le \sum_{M<n\le M+N}\sum_{u\in \cU}\left|\sum_{1\le v \le V}\chi(nu^{-1}+kv) \right| =\sum_{\lambda=1}^{q}I(\lambda)\left|\sum_{1\le v \le V}\chi(\lambda+kv) \right|,
\end{align*}
where $I(\lambda)$ counts the number of solutions to the congruence
\begin{align*}
nu^{-1}\equiv \lambda \mod{q}, \quad M<n\le M+N, \ \ u\in \cU.
\end{align*}
By H\"{o}lder's inequality 
\begin{align*}
|W|^{4}\le \left(\sum_{\lambda=1}^{q}I(\lambda) \right)^{2}\left(\sum_{\lambda=1}^{q}I(\lambda)^2 \right)\left(\sum_{\lambda=1}^{q}\left|\sum_{1\le v \le V}\chi(\lambda+kv)\right|^{4} \right).
\end{align*}
We have
\begin{align*}
\sum_{\lambda=1}^{q}I(\lambda)=N|\cU|\le NU,
\end{align*}
and, by $N<q$ and since $28\le U\le \frac{N}{12}$, by our assumptions on the range of $q$, by Lemma~\ref{lem:multcong}
\begin{align*}
I_q(N,U)\le  2UN\left(\frac{NU}{q}+\log(1.85 U) \right),
\end{align*}
using that $N\le v_1(m,q)^2 q^{\frac{5}{8}}\frac{\log q}{ \log \log q}$, we obtain
\begin{align*}
\sum_{\lambda=1}^{q}I(\lambda)^2 \le v_2(m,q)UN\left(\frac{\log {q}}{\log \log q}\right)^2.
\end{align*}
By Lemma~\ref{lem:burmv}
\begin{align*}
\sum_{\lambda=1}^{q}\left|\sum_{1\le v \le V}\chi(\lambda+kv)\right|^{4} & \le 16qk^2V^2+4q^{1/2}k^4V^4d(q)^6.
\end{align*}
Recalling~\eqref{eq:UVdef}, the above estimates simplify to
\begin{align*}
\sum_{\lambda=1}^{q}I(\lambda)\le \frac{N^2}{gq^{1/4}},
\end{align*}
\begin{align*}
\sum_{\lambda=1}^{q}I(\lambda)^2 \le v_2(m,q)\frac{N^2}{gq^{1/4}}\left(\frac{\log {q}}{\log \log q}\right)^2,
\end{align*}
\begin{align*}
\sum_{\lambda=1}^{q}\left|\sum_{1\le v \le V}\chi(\lambda+kv)\right|^{4}\le \frac{17}{4} q^{3/2}d(q)^6.
\end{align*}
Therefore
\begin{align*}
|W|^4\le \frac{17 v_2(m,q)}{4g^3} N^6q^{3/4}d(q)^6(\frac{\log {q}}{\log \log q})^2.
\end{align*}
Note that
\begin{align*}
\quad UV\ge \frac{N}{gk}\left(1-\frac{k}{q^{\frac{1}{4}}}-\frac{gq^{\frac{1}{4}}}{N}+\frac{gk}{N}\right)\ge \frac{N}{gk}\left(1-\frac{1}{h}-\frac{gq^{\frac{1}{4}}}{m^2q^{3/8}\log q (\log \log q)^2}\right).
\end{align*}
Thus using~\eqref{eq:Wbb1},~\eqref{eq:E}, Theorem 15 in \cite{Rosser} and that 
\begin{align*}
|\cU|\ge \frac{\phi(q)}{2q}U,
\end{align*}
we get 
\begin{align*}
\left|\sum_{M<n\le M+N}\psi(n)\chi(n)\right|\le v_3(m,q,g,h) kd(q)^{3/2}N^{1/2}q^{3/16}(\log q)^{\frac{1}{2}} (\log \log q)^{\frac{1}{2}},
\end{align*}
now, if $v_3(m,q,g,h)\le m$, we conclude the proof by induction.
\end{proof}
Theorem~\ref{TB} follows by computationally finding $g$ and $h$ such that for small $q$ we have a small $m$ such that $v_3(m,q,g,h)\le m$. 
\section{Explicit improved P\'{o}lya--Vinogradov }
\label{PVS}
The aim of this section is to prove Theorem \ref{P-V} following \cite{Hildebrand}.
\subsection{Two important lemmas}
We use Corollary \ref{c1} and Theorem \ref{TB} to obtain the explicit version of \cite[Lemma 2]{Hildebrand} with a certain range for the modulus $q$.
\begin{proof}~[Lemma \ref{LI1}]
Let $\epsilon$, $\chi$, $q$, $x$ and $\alpha$ be fixed and set $N= [x]$, $R=(\log q)^{\gamma}$. By $q \ge 10^{5}$ and $\gamma \ge 2$, we easily obtain $\frac{e^3}{E} \le R \le N$. By Dirichlet's theorem there exist integers $r$ and $s$, where $(r, s)=1$ and $ 1\le s \le N/R$, such that
\begin{equation}
\label{alpha}
\left| \alpha - \frac{r}{s} \right| \le \frac{1}{sN/R}.
\end{equation}
If $ s\ge R$, the result follows from Corollary \ref{c1}, since
\begin{align*}
c_1(1, E) \frac{N}{\log N}+& c_2(1,E, R) N \frac{(\log R)^{\frac{3}{2}}}{\sqrt{R}} \le \nonumber \\ &\le \left((\frac{3}{8}+\epsilon)^{-1} c_1(1,E)+\frac{c_2(1,E,R) (\gamma \log \log q)^{\frac{3}{2}}}{( \log q)^{\frac{\gamma}{2}-1}} \right) \frac{x}{\log q},
\end{align*}
by the definition of $N$ and $R$. 
Now suppose $s < R$. We obtain by partial summation 
\begin{align*}
\left| \sum_{n \le x} \chi(n) e(N'/q n) \right| \le  \left( 1+2 \pi \left| \alpha - \frac{r}{s}\right| x\right)\max_{u \le x} |T(u)| \nonumber \\ \le \left(1 + 4 \pi (\log q)^{\gamma}\right) \max_{u \le x} |T(u)|,
\end{align*}
where 
\begin{equation*}
T(u)=\sum_{n \le x} \chi(n)e\left( \frac{rn}{s} \right).
\end{equation*}
By grouping the terms of the sum $T(u)$ according to the value of $(n,s)$, we get
\begin{align*}
T(u)&=\sum_{dt=s} \sum_{\substack{dm \le u \\ (m,t)=1}} \chi(md) e\left( \frac{rm}{t}\right)\\&= \sum_{dt=s} \chi(d) \sum_{\substack{1\le a \le t\\ (a,t)=1}} \chi(md) e\left( \frac{ra}{t}\right)\sum_{\substack{m \le u/d\\ m=a (\text{mod}\quad t)}} \chi(m)\\&
=\sum_{dt=s} \frac{\chi(d)}{\phi (t)} \sum_{\psi \text{mod}\quad  t}\sum_{1\le a \le t}  e\left( \frac{ra}{t}\right)\overline{\psi}(a)\sum_{m \le u/d} \chi(m) \psi(m).
\end{align*}
Applying \eqref{Tr} to the right hand sum we obtain
\begin{align*}
\left| \sum_{n \le x} \chi(n)  e(N'/q n) \right| \le 
\frac{3 m(\log q)^{2\gamma+1} \left(1 + 4 \pi (\log q)^{\gamma}\right)(\log q \log \log q)^{\frac{1}{2}}}{q^{\epsilon/2-\frac{3 \log 2 }{2\log \log q}\left(1+\frac{1}{\log \log q}+\frac{4.7626}{(\log \log q)^2} \right)}}\frac{x}{\log q}.
\end{align*}
Thus Lemma \ref{LI1} follows.
\end{proof}

We then need explicit bounds on two trigonometric sums, by Pomerance from Lemma 2 and Lemma 3 in \cite{Pomerance}. 
\begin{lemma}
\label{LI2}
Uniformly for $x \ge 1$ and real $\alpha$ we have
\begin{equation*}
\sum_{n \le x}\frac{1-\cos (\alpha n)}{n} \le \log x +C +\log 2 +\frac{3}{x}
\end{equation*}
and
\begin{equation*}
\sum_{n \le x}\frac{|\sin (\alpha n)|}{n} \le \frac{2}{\pi} \log x +\frac{2}{\pi}\left( C +\log 2 +\frac{3}{x}\right).
\end{equation*}
\end{lemma}
Note that the last terms in the above upper bounds can be improved, but this would have no effect on our final result.
\subsection{Proof of Theorem \ref{P-V}}
We take $\chi$ primitive. We start  with
\begin{equation*}
\chi(n)=\frac{1}{d (\overline{\chi})}\sum_{a=1}^q \overline{\chi}(a) e\left( \frac{an}{q} \right) = \frac{1}{d (\overline{\chi})} \sum_{0< |a|< q/2} \overline{\chi}(a) e\left( \frac{an}{q} \right),
\end{equation*}
where $d (\overline{\chi})$ is the Gaussian sum. Summing over $ 1 \le n \le N$, we obtain
\begin{equation*}
\sum_{n=1}^N \chi(n)= \frac{1}{d (\overline{\chi})} \sum_{0< |a|< q/2} \overline{\chi}(a) \sum_{n=1}^N e\left( \frac{an}{q} \right)= \frac{1}{d (\overline{\chi})} \sum_{0< |a|< q/2} \overline{\chi}(a) \frac{ e\left( \frac{aN}{q} \right)-1}{1- e\left( \frac{-a}{q} \right)}.
\end{equation*}
Since $d (\overline{\chi})=\sqrt{q}$ for primitive characters and 
\begin{equation*}
\frac{ 1}{1- e\left( \frac{-a}{q} \right)} = \frac{q}{2 \pi i a} -\frac{\sum_2^{\infty} \frac{(-\frac{2 \pi i a}{q})^{j-2}}{j!} }{-\frac{q}{2\pi i a}\left( e\left(\frac{-a}{q}\right)-1\right) },
\end{equation*}
for $0< |a|< q/2$. 
Using $0< |a|< q/2$, it is easy to see that
\begin{equation*}
\left| \sum_2^{\infty} \frac{\left(-\frac{2 \pi i a}{q}\right)^{j-2}}{j!}\right|\le \frac{e^{\pi}-1-\pi}{\pi^2},
\end{equation*}
then, with $x=\frac{2\pi a}{q}$, we observe that
\begin{equation*}
\left|-\frac{q}{2\pi i a}\left( e\left(\frac{-a}{q}\right)-1\right)\right|=\frac{1}{|x|}\sqrt{(\cos x-1)^2+(\sin x)^2},
\end{equation*}
considering that the derivative of the right hand side is negative for $|x|\le \pi$, then
\begin{equation*}
\left|-\frac{q}{2\pi i a}\left( e\left(\frac{-a}{q}\right)-1\right)\right|> \frac{2}{\pi}.
\end{equation*}
It follows that
\begin{equation*}
\sum_{n=1}^N \chi(n)\le \frac{\sqrt{q}}{2 \pi} \left| \sum_{0< |a|< q/2}\frac{\overline{\chi(a)}\left( e(\frac{aN}{q})-1\right)}{a}\right| + \frac{\left(e^{\pi}-1-\pi\right)}{2\pi} \sqrt{q}.
\end{equation*}
Now we split the inner sum in two parts: $\Sigma_1$ with $0<|a|\le q_1=q^{\frac{3}{8}+\epsilon}$ and $\Sigma_2$ with $q_1<|a|<q/2$.\\
By partial summation, Lemma \ref{LI1} and Theorem \ref{TB} we have
\begin{align*}
\left|\Sigma_2\right | \le  2c(\chi)\left(\frac{1}{\log q}+\frac{5}{8}-\epsilon\right)c(E, q, \gamma, \epsilon,m)+1.
\end{align*}
\begin{equation*}
 \Sigma_1= 
\begin{dcases}
     2i\sum\limits_{1\le a \le q_1}\frac{\overline{\chi(a)}\sin(\frac{2\pi aN}{q})}{a} ~~\text{if}~~\chi(-1)=1,\\\\
    -2\sum\limits_{1\le a \le q_1}\frac{\overline{\chi(a)}\left( 1-\cos(\frac{2\pi aN}{q})\right)}{a}~~\text{if}~~\chi(-1)=-1,
\end{dcases}
\end{equation*}
and from Lemma \ref{LI2}
\begin{equation*}
  \left|\Sigma_1 \right| \le
  \begin{dcases}
     2 \left( \frac{2}{\pi} \log q_1 +\frac{2}{\pi}\left( C +\log 2 +\frac{3}{q_1 }\right)\right)~~\text{if}~~\chi(-1)=1,\\
    ~\\
    2\left(\log q_1 +C +\log 2 +\frac{3}{q_1}\right)~~\text{if}~~\chi(-1)=-1.
  \end{dcases}
\end{equation*}
And thus we obtain the desired result.
\section{Optimization problem}
\label{OP}
The aim of this section is to obtain a completely explicit and concise version of Theorem \ref{P-V}, thus to prove Theorem \ref{EP-V} and Tables \ref{tab:1} and \ref{tab:2}.\\
To this aim we need to optimize Theorem \ref{P-V} in the variables $\epsilon$, $q$, $E$ and $\lambda$, and in doing so we aim to minimize $\epsilon$ and $q$, and at the same time $n(q,\epsilon)$ and $m(E,q,\gamma,\epsilon)$. We will now start introducing some bounds on these variables and make some useful comments:
\begin{itemize}
\setlength\itemsep{0.01em}
\item We will use $m=0.53$, $h=300$ and $q\ge 10^{400}$ from Table \ref{tab:m}
\item Choosing $\gamma$ and a lower bound on $q$ we must ensure that $p> (300(\log q)^{\gamma})^4$
\item To minimize the second term of $c(E,q,\gamma,\epsilon)$ we need to choose $\gamma$ such that $(\log q)^2(\log \log q)^3\le (\log q)^{\gamma}$
\item Confronting Theorem \ref{P-V} with equation \eqref{eq:oldPV} we will assume $\epsilon<\frac{1}{8}$
\item The above point and the definition of $c(E,q,\gamma,\epsilon)$ implies that \\$\frac{1}{16}>\frac{3 \log 2 }{2\log \log q}\left(1+\frac{1}{\log \log q}+\frac{4.7626}{(\log \log q)^2} \right)$, which implies $q\ge e^{e^{17.82}}$
\item It is interesting to note that for any $\gamma>2$ we have, for $q \rightarrow \infty$, that
\begin{equation*}
c(E,q,\gamma,\epsilon)\longrightarrow \frac{8}{3 }c_1(1,E,(\log q)^{\gamma})
\end{equation*}
\item The above function quickly stabilises on the limit
\item Increasing $\gamma$ reduces the left hand term of $c(E,q,\gamma,\epsilon)$ and increases the right hand term
\item Choosing a small $E$ appears to be optimal
\end{itemize}
From Theorem \ref{P-V} and the above observations Tables \ref{tab:1} and \ref{tab:2} follow by computation. The optimization problem results, in this case, in a simple solution as we are forced to take $q$ big to have $h_{1/2}(E, q, \gamma, \epsilon)$ small enough, over this range of $q$ the optimal $\gamma$ is constant. We obtain that $\gamma=E=4$ and $m=0.53$ are optimal. \\
We will prove a version of Table \ref{tab:1} and \ref{tab:2} for all $q$ such that $d(q)=U$, with $U$ a fixed constant. It is easy to see, by Theorem \ref{TB} and the proof of Lemma \ref{LI1}, that in this case Theorem \ref{P-V} holds but with $d(q)=U$ instead of the general upper bound due to Robin.
In the above formula we will chose the optimal $m$ from Table \ref{tab:m}, depending on the range of $q$.
Computations now give the following table for $U=2$ and thus $q$ prime. \\We will focus on small $\epsilon$ with the aim of minimizing $q$, while keeping the constant limited. The optimization problem is harder in this case as $q$ can be taken relatively small, thus, after choosing a lower bound for $q$, we have to optimize $\gamma$ for each medium sized $q$. This means that for each medium sized $q$ we need to find the $\gamma$ that minimizes the result and then take the maximum between all of them.
To ease this problem we can balance $q$ and $h_{1/2}$ to ensure that the following result improves on \cite{F-S} in the chosen range of $q$, this will give us a $q$ big enough to make the optimization problem simpler. 
\begin{center}
\captionof{table}{$q$ prime} \label{tab:4} 
\def\arraystretch{1.25}%
\begin{tabular}{|c| c | c | c | c |c| }
  \hline
  $\epsilon$ &$\log \log q_0$  & $h_1(E, q, \gamma, \epsilon, m)$ &$\log \log q_0$ & $h_2(E, q, \gamma, \epsilon)$\\[0.5ex] 
  \hline \hline
   $\frac{1}{8}(1-\frac{1}{10})$&  13.9  & 2594 &14.1& 5183 \\
  \hline
   $\frac{1}{8}(1-\frac{1}{100})$&  16.1 & 2480 & 16.4 &4955\\
  \hline
   $\frac{1}{8}(1-\frac{1}{1000})$& 18.4 & 2469 &18.7 & 4933\\
  \hline 
   $\frac{1}{8}(1-\frac{1}{10000})$& 20.8 & 2468 &21  &4931\\
  \hline
\end{tabular}
\end{center}
It is interesting to note that in the above case, even if $h_{1/2}$ are the same as in Table \ref{tab:2}, we have lower bounds on $q$ that are significantly smaller compared to the case in which $q$ is a highly composite number, it is the size of $h_{1/2}$ that forces $q$ to be big to do better than \cite{F-S}. Thus an improvement on Corollary \ref{c1} would lead to an important improvement on the size of $q$.

\section*{Acknowledgements}
I would like to thank my supervisor Tim Trudgian and Bryce Kerr for the fundamental help in developing this paper. I would also like to thank Aleksander Simoni\v{c} for the talks we had on this paper.


\begin{thebibliography}{9}

\bibitem{Bennett}
M.~A. Bennett, G. Martin, K. O'Bryant,  and A. Rechnitzer,
\textit{Explicit bounds for primes in arithmetic progressions.}
Illinois J. Math., 62(1-4):427--532, 2018.

\bibitem{B-K}
M. Bordignon and B. Kerr,
\textit{An explicit P\'{o}lya--Vinogradov inequalyty via Partial Gaussian sums},
Arxiv, arXiv:1909.01052.


\bibitem{Burgess2}
D. A. Burgess,
\textit{On character sums and {$L$}-series},
Proc. London Math. Soc. (3), 12:193--206, 1962.

\bibitem{Burgess3}
D. A. Burgess,
\textit{The character sum estimate with {$r=3$}},
J. London Math. Soc. (2), 33(2):219--226, 1986. 

\bibitem{Burgess}
D. A. Burgess,
\textit{Partial {G}aussian sums},
Bull. London Math. Soc., 20(6):589--592, 1988.

\bibitem{Forrest}
F. J. Francis, 
\textit{ An investigation into Several Explicit Burgess Inequalities}, 
arXiv:1910.12669.

\bibitem{Frolenkov}
D. Frolenkov,
\textit{A numerically explicit version of the {P}\'{o}lya-{V}inogradov
              inequality},
Mosc. J. Comb. Number Theory, 1(3), 25--41, 2011.

\bibitem {F-S}
D. A. Frolenkov and K. Soundararajan,
\textit{A generalization of the {P}\'olya-{V}inogradov inequality.}
Ramanujan J., 31(3):271--279, 2013.
  
\bibitem{G-S}
A. Granville and K. Soundararajan,
\textit{Large character sums: pretentious characters and the
              {P}\'{o}lya-{V}inogradov theorem},
J. Amer. Math. Soc., 20(2):357--384, 2007.

\bibitem{Hildebrand}
A. Hildebrand,
\textit{On the constant in the {P}\'{o}lya-{V}inogradov inequality},
Canad. Math. Bull., 31(3):347--352, 1988.

\bibitem{Hild1}
A. Hildebrand,
\textit{Large values of character sums},
J. Number Theory, 29(3):271--296, 1988.

\bibitem{MV1}
H. L. Montgomery and R. C. Vaughan, 
\textit{The large sieve},
Mathematika, 20:119--134, 1973.
  
\bibitem{M-V}
H. L. Montgomery and R. C. Vaughan, 
\textit{Exponential sums with multiplicative coefficients},
Invent. Math., 43(1):69--82, 1977.

 \bibitem{HB}
D. R. Heath-Brown, {\it Hybrid Bounds for Dirichlet L-functions}, Invent. Math. {\bf 47}, 149--170, 1978.

 \bibitem{Korobov}
N. M. Korobov,, {\it Exponential sums and their applications}, Kluwer Academic Publishers Group, 1992.

\bibitem{Paley}
R. E. A. C. Paley, {\it A theorem on characters}, J. London Math. Soc., {\bf 7}, (1932), 28--32. 

\bibitem{Pomerance}
C. Pomerance,
\textit{Remarks on the {P}\'{o}lya-{V}inogradov inequality},
Integers, 11(4):531--542, 2011. 

\bibitem{Robin}
G. Robin,
\textit{Th\'{e}se d'\'{e}tat},
Universit\'{e} de Limoges, France, (1983). 

\bibitem{Rosser}
J. Rosser and L. Schoenfeld,
\textit{Approximate formulas for some functions of prime numbers},
Illinois J. Math., (6):64--94, 1962. 

\bibitem{Sandor}
J. S\'andor, D. S. Mitrinovi\'c and B. Crstici,
\textit{Handbook of Number Theory {I}}
Springer, 2006.
 
\bibitem{Siebert}
H. Siebert,
\textit{Montgomery's weighted sieve for dimension two},
Monatsh. Math., 82(4):327--336, 1976.

\bibitem{Trev}
E. Trevi\~no,
\textit{The {B}urgess inequality and the least {$k$}th power
              non-residue.}
Int. J. Number Theory, 11(5):1653--1678, 2015.

\end{thebibliography}
\end{document}